\documentclass[12pt]{amsart}

\usepackage{amsmath, amsthm, amssymb, amsfonts, mathrsfs, amscd, tikz, verbatim, cite}
\usepackage{enumerate, paralist}
\usetikzlibrary{matrix,arrows}

\def\QQ{{\mathbb Q}}
\def\PP{{\mathbb P}}
\def\QQ{{\mathbb Q}}

\def\ZZ{{\mathbb Z}}

\def\0{{\mathbf 0}}
\def\1{{\mathbf 1}}

\def\Mcal{{\mathcal M}}

\def\Ocal{{\mathcal O}}

\def\Scal{{\mathcal S}}

\def\Vcal{{\mathcal V}}
\def\Wcal{{\mathcal W}}

\def\afrak{{\mathfrak a}}

\def\pfrak{{\mathfrak p}}

\def\Rfrak{{\mathfrak R}}

\def\Kbar{{\bar K}}

\def\Aut{\mathrm{Aut}}

\def\PGL{\mathrm{PGL}}

\def\Res{\mathrm{Res}}
\def\Hom{\mathrm{Hom}^n_d}

\def\min{\mathrm{min}}
\def\uf{\mathrm{uf}}
\def\uc{\mathrm{uc}}

\def\Norm{\mathrm{N}}

\def\Ratfd2{\mathrm{Rat}_{d,2}^{\uf}}
\def\Mfd2{\mathcal{M}_{d,2}^{\uf}}
\def\Rfd2{\mathcal{R}_{d,2}^{\uf}}
\def\Rcd2{\mathcal{R}_{d,2}^{\uc}}

\def\vp{\varphi}
\def\Pn{\PP^n}

\def\Homdn{\mathrm{Hom}_d^n}
\def\Ratdn{\mathrm{Rat}_d^n}

\def\Mdn{\mathcal{M}^n_d}

\def\Tw{\mathrm{Twist}}

\theoremstyle{plain}

\newtheorem{thm}{Theorem}

\newtheorem{prop}[thm]{Proposition}
\newtheorem{lem}[thm]{Lemma}

\theoremstyle{definition}
\newtheorem{dfn}[thm]{Definition}

\newtheorem{rem}{Remark}

\begin{document}

\title[Bounded Height and Resultant]{Endomorphisms of Bounded Height and Resultant}

\author{Brian Stout}
\author{Adam Towsley}

\address{Brian Stout; Department of Mathematics; United States Naval Academy;  Annapolis, MD 21401 U.S.A.}
\email{stout@usna.edu}

\address{Adam Towsley; Department of Mathematics; The City University of New York Graduate Center; New York, NY 10016 U.S.A.}
\email{atowsley@gc.cuny.edu}

\thanks{{\em Date of last revision:} May 8, 2014}

\begin{abstract}
Let $K$ be an algebraic number field and $B\geq 1$.  For an endomorphism $\vp:\PP^n\rightarrow\PP^n$ defined over $K$ of degree $d$ let $\Rfrak_\vp\subset\Ocal_K$ denote its minimal resultant ideal.  For a fixed height function $h_{\Mdn}$ on the moduli space of dynamical systems this paper shows that all such morphisms $\vp$ of bounded resultant $\Norm_{K/\QQ}(\Rfrak_\vp)\leq B$ and bounded height $h_{\Mdn}(\langle\vp\rangle)\leq B$ are contained in finitely many $\PGL_{n+1}(K)$-equivalence classes. This answers a question of Silverman in the affirmative.
\end{abstract}
\maketitle


\section{Introduction}\label{Introduction}
In \cite{SilvermanADS} Silverman asked the following question: 
given a rational map $\vp: \PP^1 \rightarrow \PP^1$ of degree $d$ defined over $K$ are there only finitely many $K$-isomorphism classes of $\vp$ of bounded height and bounded resultant? 
We prove the following $n$-dimensional generalization of Silverman's question:
\begin{thm}\label{MainTheorem}
 Let $K$ be a number field and let $\vp: \Pn \rightarrow \Pn$ be a morphism of degree $d\geq 2$ defined over $K$. Let $\Rfrak_\vp$ denote the minimal resultant ideal of $\vp$. Fix an embedding of 
 $\Mdn$ into projective space and denote the associated height function by $h$. If $B \geq 1$ then the set
 $$\Gamma_{K,B} = \left\lbrace \vp \in \Homdn \left( K \right) \ : \ \Norm_{K/\QQ}( \Rfrak_\vp) \leq B \text{ and } h \left( \left< \vp \right> \right) \leq B \right\rbrace$$
 is contained in only finitely many $\PGL_{n+1} \left( K \right)$-conjugacy classes of morphisms.
\end{thm} 
The minimal resultant ideal is a way of encoding information about the primes of bad reduction for $\vp$ across its conjugacy class. 
For every non-archimedean place $v$ for which $\Rfrak_\vp$ has positive valuation, some conjugate $\vp^f$ for $f\in\PGL_{n+1}(K)$ has bad reduction.  

One can consider the height $h(\langle\vp\rangle)$ as a measure of the arithmetic complexity of the conjugacy class of $\vp$ and the norm of the resultant as a measure of the amount of bad reduction.  As we will see in Section \ref{Bounded}, bounding the norm of the minimal resultant ideal bounds the primes for which $\vp$ has bad reduction.  

Bounding both the height of the conjugacy class and the norm of the minimal resultant ideal are needed to obtain a finiteness result.  If only the norm $\Norm_{K/\QQ}(\Rfrak_\vp)$ is bounded, then one can still find infinitely many distinct $\PGL_{n+1}(K)$-conjugacy classes of endomorphisms defined over $K$. This can be accomplished, for example, by considering monic polynomials defined over $\Ocal_K$. Monic polynomials have everywhere good reduction and hence their minimal resultant ideal is the unit ideal, which has norm $1$. It is easy to show that there are infinitely many distinct $\PGL_{n+1}(\Kbar)$-conjugacy classes of such maps, and therefore infinitely many distinct $\PGL_{n+1}(K)$-conjugacy classes.  

Conversely, if one bounds only the height $h(\langle\vp\rangle)$ one gets a finite set of points in $\Mdn(K)$. 
Points in $\Mdn(K)$ are the same as $\PGL_{n+1}(\Kbar)$-conjugacy classes of endomorphisms defined over $K$; 
denote these classes by $[\vp_1],\ldots, [\vp_r]$ for endomorphisms $\vp_i$ of degree $d$ on $\PP^n$ defined over $K$.  
Each $\PGL_{n+1}(K)$-conjugacy class  which descends to some $[\vp_i]$ is called a twist.  
A priori, it is not obvious how many twists over $K$ descend to each $[\vp_i]$. 
Indeed, if the automorphism group of $\vp$, which consists of the finite subgroup of $\PGL_{n+1} \left( \bar{K} \right)$ that fixes $\vp$ under conjugation, is trivial, then there is only one twist. This need not be the case in general. For an example where there are infinitely many twists one can consider quadratic rational maps on $\PP^1$ of the following form
\begin{equation*}
\vp_b(z)=z+\dfrac{b}{z}
\end{equation*}
where $b\in K^*$.  All $\vp_b$ are $\Kbar$-isomorphic and therefore descend to the point $[\vp_1]\in\Mcal_2$. However, $\vp_b$ and $\vp_c$ are $K$-isomorphic if and only if $b/c$ is a square in $K$.  If we denote the set of twists of $\vp_1$ by $\Tw(\vp_1/K)$, then this gives an injective map $K^*/K^{*2}\rightarrow\Tw(\vp_1/K)$ by $b\mapsto [\vp_b]_K$. See example 4.71 in \cite{SilvermanADS} for more details.  For a number field $K^*/K^{*2}$ is infinite (for example, over $\QQ$ this set contains all the primes), so it follows that bounding the height $h(\langle\vp\rangle)$ is not enough to guarantee a finiteness result.

\emph{Acknowledgements}   
Both authors would like to thank Tom Tucker for reading this paper and his helpful comments, and the referee for several helpful comments.

\section{Preliminaries}\label{Prelims}
\subsection{The spaces $\Hom$ and $\Mdn$}

Let $K$ be a number field and $\Kbar$ be a fixed algebraic closure.  By assumption, varieties and morphisms will be defined over $\Kbar$. When we wish to emphasize a special field of definition, we will use the notation $\PP^n(K)$.
Let $$\Homdn = \left\lbrace \vp:\Pn\rightarrow\Pn\text{ is a degree } d \text{ morphism} \right\rbrace.$$ 
By degree we mean algebraic degree. We will see that $\Hom$ is a variety over $\Kbar$ and that $K$-rational points correspond to morphisms of degree $d$ on $\Pn$ defined over $K$. Equivalently, points of $\Hom(K)$ correspond to endomorphism $\vp:\PP^n\rightarrow\PP^n$ of algebraic degree $d$ defined over $K$.

After fixing a basis $X_0,\ldots ,X_n$ of $\Pn$, any morphism $\vp \in \Homdn$ can be written as $\vp = \left[ \vp_0, \dots, \vp_n \right]$, where each $\vp_i$ is a degree $d$ homogeneous polynomial in the $X_i$, and $\vp_0, \dots, \vp_n$ have no non-trivial common zeros over the algebraic closure $\Kbar$.
Each $\vp_i$ can be written as \begin{equation*}\vp_i = \displaystyle\sum_I a_I X^I\end{equation*} with multi-index $I$. Here, $I=(i_0,\ldots,i_n)$ and $i_0+\cdots +i_n=d$ and $X^I=X^{i_0}_0\cdots X^{i_n}_n$. We order these monomials using the lexographic ordering.
Thus, if $N = \left({n+d \choose d} \right) \left( n+1 \right) - 1$ is the number of monomials $X^I$ then we can identify any $\vp \in \Homdn$ with a point in $\PP^N$ via the association
$\vp \mapsto \left[a_I\right] \in \PP^N$. It follows that $\Hom(K)\subset\PP^N(K)$.

\begin{thm}[Theorem 1.8 in \cite{SilvermanBarbados}] \label{Silverman1.8}
There exists a geometrically irreducible polynomial $\Res \in \ZZ\left[a_i\right]$ in the coefficients of $\vp$ such that
$$ \vp \in \Homdn \Leftrightarrow \Res \left( \vp \right) \neq 0.$$
\end{thm}

The polynomial $\Res$ is called the Macaulay resultant of $\vp$. We remark that we abuse notation above.  Let $\Phi$ be an choice of the coordinates of the projective point defined by $\phi$.  If $\Psi$ is another choice, then $\Phi=\lambda\Psi$ some $\lambda\in\Kbar$. It follows from elementary properties of Macaulay resultant that $\Res(\Phi)=\lambda^{(n+1)d^n}\Res(\Psi)$.  Therefore, only the vanishing of $\Res$ on $\PP^N$ is well defined. It is multi-homogeneous in the coefficients of the $\vp_i$. This implies that $\Homdn$ is an affine variety. Specifically, $\Hom=\PP^N-V(\Res)$. For further discussion and proof of these facts regarding the Macaulay resultant see \cite{CoxLittleOshea}.

The automorphism group $\Aut \left( \Pn \right) = \PGL_{n+1}\left( \bar{K} \right)$ acts on $\Homdn$, and on $\Ratdn$ (the set of rational functions from $\Pn$ to itself) 
via the action of conjugation, indeed it acts on the entire space $\PP^N$; for $\vp \in \Homdn$ and $f \in \PGL_{n+1}\left( \bar{K} \right)$, $\vp^f = f^{-1}\circ \vp \circ f$. 
It follows from the properties of the resultant that $\vp\in\Hom\Leftrightarrow\vp^f\in\Hom$, so $\Hom$ is a $\PGL_{n+1} \left( \bar{K} \right)$-stable subset for the action of this group on $\PP^N$. 
We then define the moduli space $\Mdn$ to be the set of all conjugacy classes of endomorphisms of $\Pn$, that is
$$\Mdn = \Homdn / \PGL_{n+1} \left( \bar{K} \right).$$ 
For any $\vp \in \Homdn$ we denote its image in $\Mdn$ by $\left< \vp \right>.$ 
It follows from Geometric Invariant Theory that $\Mdn$ exists as an affine variety (see  \text{section 2.3 in }\cite{SilvermanBarbados}).

By picking an appropriate $M$ the moduli space $\Mdn$ can be embedded into $\mathbb{P}^M$. 
It follows from Theorems 2.24 and 2.26 in \cite{SilvermanBarbados} that there exists a projective variety ${\Mdn}^{,ss}$, the moduli space of semi-stable dynamical systems on $\Pn$. 
The variety ${\Mdn}^{,ss}$ exists as the quotient of the semi-stable locus for the $\PGL_{n+1} \left(\bar{K}\right)$ action on $\PP^N$ and it follows that $\Mdn\subset{\Mdn}^{,ss}$ is a dense open subset. As ${\Mdn}^{,ss}$ is projective, it embeds into some $\PP^M$ as does $\Mdn$ by the inclusion $\Mdn\hookrightarrow{\Mdn}^{,ss}\hookrightarrow\PP^M$. We can therefore define a height function
$h_{\Mdn}$ on $\Mdn$ by taking the height of the image of $\left< \vp \right>$ in $\mathbb{P}^M$. We will denote this height function by $h$.

\subsection{Minimal Resultants}\label{SubSection_MinResultant}
Following \cite{SilvermanBarbados} we now use the Macaulay resultant to define the minimal resultant of a an endomorphism.

Let $R$ be a discrete valuation ring with discrete valuation $v$ and field of fractions $F$. If $\vp= \left[ \vp_1, \dots, \vp_n \right] \in \Homdn \left( F \right)$ 
then we define 
\begin{equation*}
e_v (\vp)= v \left( \Res \left( \vp \right) \right) - \left( n+1 \right) d^n \underset{0\leq i\leq n}{\min} v \left( \vp_i \right).
\end{equation*}
Where $v \left( \vp_i \right)$ is taken to be the minimal valuation of the coefficients of $\vp_i$.  We remark that $e_v(\vp)$ is well defined.  Let $\Phi,\Psi$ be affine models of $\vp$ and $\lambda\in F^\times$ such that $\Phi=\lambda\Psi$.  Then $v(\Res(\Phi))=v(\Res(\lambda\Psi)=v(\lambda^{(n+1)d^n}\Res(\Psi))=(n+1)d^nv(\lambda)+v(\Res(\Psi))$
Similarly, \begin{equation*}\underset{0\leq i\leq n}{\min}\text{ }v(\Phi_i)=\underset{0\leq i\leq n}{\min}\text{ }v(\lambda\Psi_i)=v(\lambda)+ \underset{0\leq i\leq n}{\min}\text{ }v(\Psi_i)\end{equation*} It follows that $e_v(\vp)$ is independent of choice of affine model used to compute it.

The quantity $e_v(\phi)$ is then used to define the \emph{exponent of the minimal discriminant of} $\vp$ by:
\begin{dfn}\label{MinResultantExponent}
\begin{equation*}
\varepsilon_v \left( \vp \right)  = \underset{f \in \PGL_{n+1}{\left( K \right)}}{\min} e_v \left( \vp^f \right).
\end{equation*}
\end{dfn}

A fact that will be important in our proof of Theorem \ref{MainTheorem} is the following:
\begin{prop}\label{GoodReduction}
 $\vp$ has good reduction if and only if $\varepsilon_v \left( \vp \right) = 0$.
\end{prop}
\begin{proof}
 See proposition 3.11 in \cite{SilvermanBarbados}. 
\end{proof}

Now let $R$ be a Dedekind domain with field of fractions $K$. 
If $v$ is a discrete valuation on $R$ then we denote the prime ideal associated to $v$ by $\pfrak_v$. 

\begin{dfn}\label{MinResultant}
 For an endomorphism $\vp \in \Homdn \left( K \right)$ the \emph{minimal resultant} of $\vp$ is
 $$\mathfrak{R}_\vp = \prod_v \pfrak_v^{\varepsilon_v \left( \vp \right)}.$$
 Where the product is taken over all of the inequivalent valuations on $R$. 
\end{dfn}

It is clear that this is an ideal of $R$ as any morphism $\vp$ has only finitely many primes of bad reduction.  It follows from Proposition \ref{GoodReduction} that $\varepsilon_v(\vp)=0$ for all but finitely many $v$.  

\subsection{Twists}

Two morphisms $\vp, \psi \in \Homdn \left( K \right)$ are $\bar{K}$-isomorphic if there is an $f \in \PGL_{n+1}\left( \bar{K}\right)$ such that $\vp  = \psi^f$, that is if $\left< \vp \right> = \left< \psi \right>$.
We denote this set of $\bar{K}$-isomorphic morphisms as 
$$\left[\vp \right] = \left\lbrace \psi \in \Homdn\left(K \right) : \psi = \vp^f \text{ for some } f \in \PGL_{n+1} \left(\bar{K}\right) \right\rbrace.$$
If we restrict our automorphisms to those defined over $K$ we get a second set
$$\left[\vp \right]_K = \left\lbrace \psi \in \Homdn\left(K \right) : \psi = \vp^f \text{ for some } f \in \PGL_{n+1} \left(K\right) \right\rbrace.$$
The twists of $\vp$ are the $K$-isomorphism classes of $\vp$ which are $\bar{K}$-isomorphic.
$$\text{Twist}_K\left(\vp\right) = \left\lbrace \left[ \psi \right]_K : \psi \in \Homdn \left( K \right) \text{ and } \left[\psi \right] = \left[ \vp \right] \right\rbrace.$$   
One can interpret twists in terms of moduli spaces.  A $\Kbar$-isomorphism class for $\vp$ defined over $K$ corresponds to a $K$-rational point of $\Mdn$.  
One can also consider the quotient variety $$\mathcal{M}^n_{d,K}=\Hom(K)/\PGL_{n+1}(K)$$ and a $K$-isomorphism class $\left[\vp\right]_K$ corresponds to a $K$-rational point of this space.  
The natural inclusion $\PGL_{n+1}(K)\hookrightarrow\PGL_{n+1}(\bar{K})$ induces a map on the quotients $\mathcal{M}^n_{d,K}\rightarrow\Mdn$.  Fibers of this morphism over the $K$-rational point $\left[\vp\right]$ are twists.

In \cite{Stout}, the first author proved the following theorem which is central in our proof of \ref{MainTheorem}:

\begin{thm}\label{FiniteTwists}
Let $\vp \in \Homdn \left(K \right)$ for $d \geq 2$, and let $S$ be a finite set of places containing the archimidean places. If
$$\Vcal \left( S \right) = \left\lbrace \left[ \psi \right]_K  \in \text{Twist}_K \left( \vp \right) : \left[ \psi \right]_K \text{ has good reduction outside } S \right\rbrace,$$ 
then $\Vcal \left( S \right)$ is finite.
\end{thm}

\section{Bounded height and resultant}\label{Bounded}
Let $h$ denote a fixed height on $\Mdn$ corresponding to some embedding $\Mdn\hookrightarrow\PP^M$. Let $K$ denote a fixed number field; we assume all $\vp$ are defined over $K$.

\begin{lem}\label{FiniteKbarPoints}
Let $d\geq 1$ be an integer and $B\geq 1$.  The set 
\begin{equation*}\Wcal =\lbrace \langle\vp\rangle | \vp\in\Homdn(K)\text{ and }h(\langle\vp\rangle)\leq B\rbrace\subset\Mdn\end{equation*} 
is finite.
\end{lem}

\begin{proof}
We claim that this set is a set of bounded height and degree in $\Mdn$.  First, every $\vp\in\Wcal$ is a $K$-rational point of $\Mdn$. This follows from the assumption that $\vp$ is defined over $K$.  Therefore, it is a $K$-rational point of $\Homdn$.  It follows via the quotient map $\Homdn\rightarrow\Mdn$ that $\vp$ corresponds to a $K$-rational point of $\Mdn$.  Hence, $\Wcal$ is a set of bounded degree.  By the assumption that each $\vp\in\Wcal$ satisfies $h(\langle\vp\rangle)\leq B$ we have that $\Wcal$ is a set of bounded height.   By the theorem of Northcott (Theorem 3.7 of \cite{SilvermanADS}), such sets of bounded degree and height are finite.
\end{proof}

We remark that this does not prove our main theorem, as this shows that $\Gamma$ is contained in finitely many $\PGL_{n+1}(\Kbar)$ conjugacy classes.  We make the following definition based on the requirement that the resultant is bounded.

\begin{dfn}
Let $B\geq 1$. We define the set of primes $\Scal_B$ of $K$ in the following way
\begin{equation*}
\Scal_B=\lbrace\pfrak |\Norm_{K/\QQ}(\pfrak)\leq B\rbrace
\end{equation*}
\end{dfn}

\begin{lem}\label{SBFinite}
Let $B\geq 1$. Then the set $\Scal_B$ of primes of $K$ is finite.
\end{lem}

\begin{proof}
There are only finitely many primes $p$ of $\ZZ$ such that $p\leq B$.  It follows from the property of the norm that
\begin{equation*}
\Norm_{K/\QQ}(\pfrak)=\underset{p|\pfrak}{\prod}p^{[K/\pfrak:\ZZ/p]}
\end{equation*}
and the fact that only finitely many primes $\pfrak$ lie over $p$ that $\Scal_B$ is finite.
\end{proof}

\begin{lem}\label{GoodRed}
Let $\Gamma_{K,B}$ be the set from Theorem \ref{MainTheorem}. If $\vp\in\Gamma_{K,B}$, then $\vp$ has good reduction at all primes $\pfrak\not\in \Scal_B$.
\end{lem}

\begin{proof}
Let $\vp\in\Gamma_{K,B}$ and $\pfrak\not\in \Scal_B$. By definition of the set $\Scal_B$, we have that $\Norm_{K/\QQ}(\pfrak)>B$. By Proposition \ref{GoodReduction} $\vp$ has bad reduction at $\pfrak$ if and only if $\pfrak|\Rfrak_\vp$. If $\vp$ has bad reduction at $\pfrak$, then $\Rfrak_\vp=\pfrak\afrak$ and hence $\Norm_{K/\QQ}(\Rfrak_\vp)=\Norm_{K/\QQ}(\pfrak)\Norm_{K/\QQ}(\afrak)>B\Norm_{K/\QQ}(\afrak)>B$. This is contrary to the assumption that $\vp\in\Gamma_{K,B}$. It follows that $\vp$ has good reduction at $\pfrak$.
\end{proof}

\begin{rem}
$\Scal_B$ may contain some primes at which $\vp$ has good reduction. Consider, for example, an endomorphism defined over $\QQ$ with minimal resultant $\Rfrak_\vp=(2)^3$. Then $\Scal_B=\lbrace 2,3,5,7\rbrace$ even though $\vp$ has bad reduction only at $2$.  
\end{rem}

We now prove theorem \ref{MainTheorem}.
\begin{proof}
It follows from Lemma \ref{FiniteKbarPoints} that $\Gamma$ descends to only finitely many point in $\Mdn(\Kbar)$, i.e. that $\Wcal=\pi(\Gamma)$ is finite under the image of the quotient map $\pi:\Homdn(K)\rightarrow\Mdn(\Kbar)$.

Let $\Vcal$ be the set of $K$-conjugacy classes of maps in $\Gamma$. 
It follows that $\Vcal\subset\Mcal^n_{d,K}$ and that under $\theta:\Mcal^n_{d,K}\rightarrow\Mdn$, $\theta(\Vcal)=\Wcal$. 
By definition, the $K$-conjugacy classes of $\Vcal$ consist of all twists lying over the finitely many $\Kbar$-conjugacy classes in $\Wcal$. If $\Vcal$ is infinite, then by the pidgeon hole principle there is at least one $\bar{K}$-conjugacy class which contains infinitely many $K$-conjugacy classes. Denote this $\Kbar$-congugacy class by $[\psi]$ and the infinitely many distinct $K$-conjugacy classes by $[\psi_i]_K$ for $i=0,1,\ldots$ where $[\psi_i]=[\psi]$ for all $i$. It follows that each $[\psi_i]_K$ is a twist of $[\psi]$.

By Lemma \ref{SBFinite} the set $\Scal_B$ is finite and by Lemma \ref{GoodRed} we see that each $[\psi_i]_K$ must have good reduction outside of $\Scal_B$. Thus we have constructed an infinite set of twists $[\psi_i]_K$ of $[\psi]$ with good reduction outside of $\Scal_B$ which contradicts Theorem \ref{FiniteTwists}. This concludes the proof of the main theorem.
\end{proof}


\bibliography{BoundedHeight}
\bibliographystyle{plain}

\end{document}